\documentclass[12pt]{amsart}
\usepackage{epsfig,color}

\makeatother

\theoremstyle{definition}

\def\fnum{equation} 
\newtheorem{Thm}[\fnum]{Theorem}
\newtheorem{Cor}[\fnum]{Corollary}

\newtheorem{Lem}[\fnum]{Lemma}
\newtheorem{Con}[\fnum]{Conjecture}

\newtheorem{Rem}[\fnum]{Remark}
\newtheorem{Pro}[\fnum]{Proposition}

\numberwithin{equation}{section}

\newcommand{\Vol}{{\text{Vol}}}

\newcommand{\nn}{{\bf{n}}}

\newcommand{\dist}{{\text {dist}}}

\newcommand{\Hess}{{\text {Hess}}}

\def\RR{{\bold R}}

\def\SS{{\bold S}}

\newcommand{\dv}{{\text {div}}}
\newcommand{\e}{{\text {e}}}

\newcommand{\graph}{{\bf{r}}_{\ell}}

\newcommand{\cC}{{\mathcal{C}}}

\newcommand{\cL}{{\mathcal{L}}}

\newcommand{\eqr}[1]{(\ref{#1})}
\newcommand{\eps}{\epsilon}

\title[\L{}ojasiewicz inequalities and applications]{\L{}ojasiewicz inequalities and applications}

\author{Tobias Holck Colding}%
\address{MIT, Dept. of Math.\\
77 Massachusetts Avenue, Cambridge, MA 02139-4307.}
\author{William P. Minicozzi II}%

\thanks{The  authors
were partially supported by NSF Grants DMS  11040934, DMS 1206827,  and NSF FRG grants DMS 
 0854774 and DMS 0853501}

\email{colding@math.mit.edu and minicozz@math.mit.edu}

\begin{document}

\maketitle
 
\begin{abstract}
In real algebraic geometry, Lojasiewicz's theorem asserts that any integral curve of the gradient flow of an analytic function that has an accumulation point has a unique limit.   Lojasiewicz proved this result in the early 1960s as a consequence of his gradient inequality.  

Many problems in calculus of variations are questions about critical points or gradient flow lines of an infinite dimensional functional.    Perhaps surprisingly, even blowups at singularities of many nonlinear PDE's can, in a certain sense, be thought of as limits of infinite dimensional gradient flows of analytic functionals.  
Thus, the question of uniqueness of blowups can be approached as an infinite dimensional version of Lojasiewicz's theorem.
The question of uniqueness of blowups is perhaps the most fundamental question about singularities.

 This approach to uniqueness was pioneered by Leon Simon thirty years ago 
 for the area functional and many related functionals
 using an elaborate reduction to a finite dimensional setting where Lojasiewicz's arguments applied.

Recently, the authors proved two new infinite dimensional Lojasiewicz inequalities at noncompact singularities where it was well-known that a reduction to Lojasiewicz's arguments is not possible, but instead entirely new techniques are required.  As a consequence, the authors settled a major long-standing open question about uniqueness of blowups for mean curvature flow (MCF) at all generic singularities and for mean convex MCF  at all singularities.   Using this, the authors have obtained a rather complete description of the space-time singular set for MCF with generic singularities.  In particular, the singular set of  a MCF in $\RR^{n+1}$ with only generic singularities is contained in finitely many compact Lipschitz submanifolds of dimension at most $n-1$ together with a set of dimension at most $n-2$.
\end{abstract}

\section{Finite and infinite dimensional inequalities}

  \subsection{Lojasiewicz inequalities}

    In real algebraic geometry, the Lojasiewicz inequality, \cite{L1}, \cite{L2}, \cite{L4}, from the late 1950s named after Stanislaw Lojasiewicz, gives an upper bound for the distance from a point to the nearest zero of a given real analytic function. Specifically, let $f: U \to \RR$ be a real-analytic function on an open set $U$ in $\RR^n$, and let $Z$ be the zero locus of $f$. Assume that $Z$ is not empty. Then for any compact set $K$ in $U$, there exist $\alpha\geq 2$ and a positive constant $C$ such that, for all $x \in K$
  \begin{align}	\label{e:Lj1}
  \inf_{z\in Z} |x-z|^{\alpha}\leq C\, |f(x) | \, .
  \end{align}
Here $\alpha$ can be large.   

 Equation \eqr{e:Lj1} was the main inequality in Lojasiewicz's proof of Laurent Schwarz's division conjecture\footnote{L. Schwartz conjectured that if $f$ is a non-trivial real analytic function and $T$ is a distribution, then there exists a distribution $S$ 
satisfying $f\, S$ = $T$.} in analysis.  
Around the same time, 
H\"ormander, \cite{Ho}, independently proved Schwarz's division conjecture in the special case of polynomials and a  key step in his proof was also 
 \eqr{e:Lj1}  when $f$ is a polynomial.

A few years later,  Lojasiewicz solved a conjecture of Whitney\footnote{Whitney conjectured that if $f$ is analytic  in an 
open set $U$ of $\RR^n$, then
 the zero set $Z$ is a deformation retract of an open neighborhood of $Z$ in 
$U$.} in \cite{L3} using
 the following inequality\footnote{Lojasiewicz called this inequality the gradient inequality.}: With the same assumptions on $f$, for every $p\in U$, there are a possibly smaller neighborhood $W$ of $p$ and constants $\beta \in (0,1)$ and $C> 0$ such that for all $x\in W$
 \begin{align}	\label{e:016}
  |f(x)-f(p)|^{\beta}\leq C\, |\nabla_x f | \, .
  \end{align}
  Note that this inequality is trivial unless $p$ is a critical point for $f$.   
  
 One immediate consequence of \eqr{e:016} is that every critical point of $f$
  has a neighborhood where 
  every other critical point has the same value.     It is easy to construct smooth functions where this is not the case.

  \subsection{First Lojasiewicz  implies the second}  \label{ss:1to2}

In this subsection, we will explain how the second Lojasiewicz inequality for a 
function $f$ in a neighborhood of an isolated critical point
follows from the first.  To make things concrete, we will show that the second holds
 with $\beta=\frac{2}{3}$
when the first holds    with $\alpha=2$.   

Suppose that $f:\RR^n\to \RR$ is  smooth function with $f(0)=0$ and $\nabla f(0)=0$; 
without loss of generality, we may assume that  the Hessian is in diagonal form at $0$
and we will write the coordinates as $x=(y,z)$ where $y$ are the coordinates where the 
Hessian is nondegenerate.  By Taylor's formula in a small neighborhood of $0$, we have that
\begin{align}
f(x)&=\frac{a_i}{2}\, y_i^2+O(|x|^3)\, .\\
f_{y_i}(x)&=a_i\, y_i+O(|x|^2)\, .\\
f_{z_i}(x)&=O(|x|^2)\, .
\end{align}
It follows from this that the second of the two Lojasiewicz 
inequalities holds for $f$ and $\beta =\frac{2}{3}$ provided that $|z|^2\leq \epsilon\, |y|$ 
for some sufficiently small $\epsilon> 0$.  Namely, if $|z|^2\leq \epsilon\, |y|$, then
\begin{align}
  C \, |y| \leq	 |\nabla_x f| {\text{ and }} |f(x)|\leq C^{-1} \,|y|^{\frac{3}{2}}
\end{align}
for some positive constant $C$ and, hence,
\begin{align}
|f(x)|^{\frac{2}{3}}\leq C\,|\nabla_x f |\, .
\end{align}
Therefore, we only need 
to prove the second Lojasiewicz inequality for $f$ 
in the region  $|z|^2\geq \epsilon\, |y|$.  We will do this
using
 the first  Lojasiewicz inequality for $\nabla f$. 
Since $0$ is an isolated critical point for $f$, 
the first Lojasiewicz inequality for $\nabla f$ gives that 
\begin{align}
|\nabla_x f|\geq C\, |x|^2\, .
\end{align}
By assumption on the region and the Taylor expansion for $f$, we get that in this region 
\begin{align}
|f(x)|\leq C\,|y|^2+C\,|z|^3\leq C\,|z|^3\leq C\,|x|^3\, .
\end{align}
 Combining these two inequalities gives
 \begin{align}
|f(x)|^{\frac{2}{3}}\leq C\,|x|^2\leq |\nabla_x f|\, .
\end{align}
This proves the second Lojasiewicz inequality for $f$ with $\beta = \frac{2}{3}$.

\vskip2mm
 Lojasiewicz used his second inequality to show the ``Lojasiewicz theorem'':  If $f:\RR^n\to \RR$ is an analytic function, $x=x(t):[0,\infty)\to \RR^n$ is a curve with $x'(t)=-\nabla f$ and $x(t)$ has a limit point $x_{\infty}$, then the length of the curve is finite and 
 \begin{align}
 	\lim_{t\to \infty}x(t)=x_{\infty} \, .
\end{align}
  Moreover, $x_{\infty}$ is a critical point for $f$.
  
  In contrast, it is easy to construct smooth functions, even on $\RR^2$, where the Lojasiewicz theorem fails, i.e., where there are negative gradient flow lines  that have more than one limit point (and, thus, also have infinite length); see Figure \ref{f:1}.
  
   \begin{figure}[htbp]		
\centering\includegraphics[totalheight=.4\textheight, width=1\textwidth]{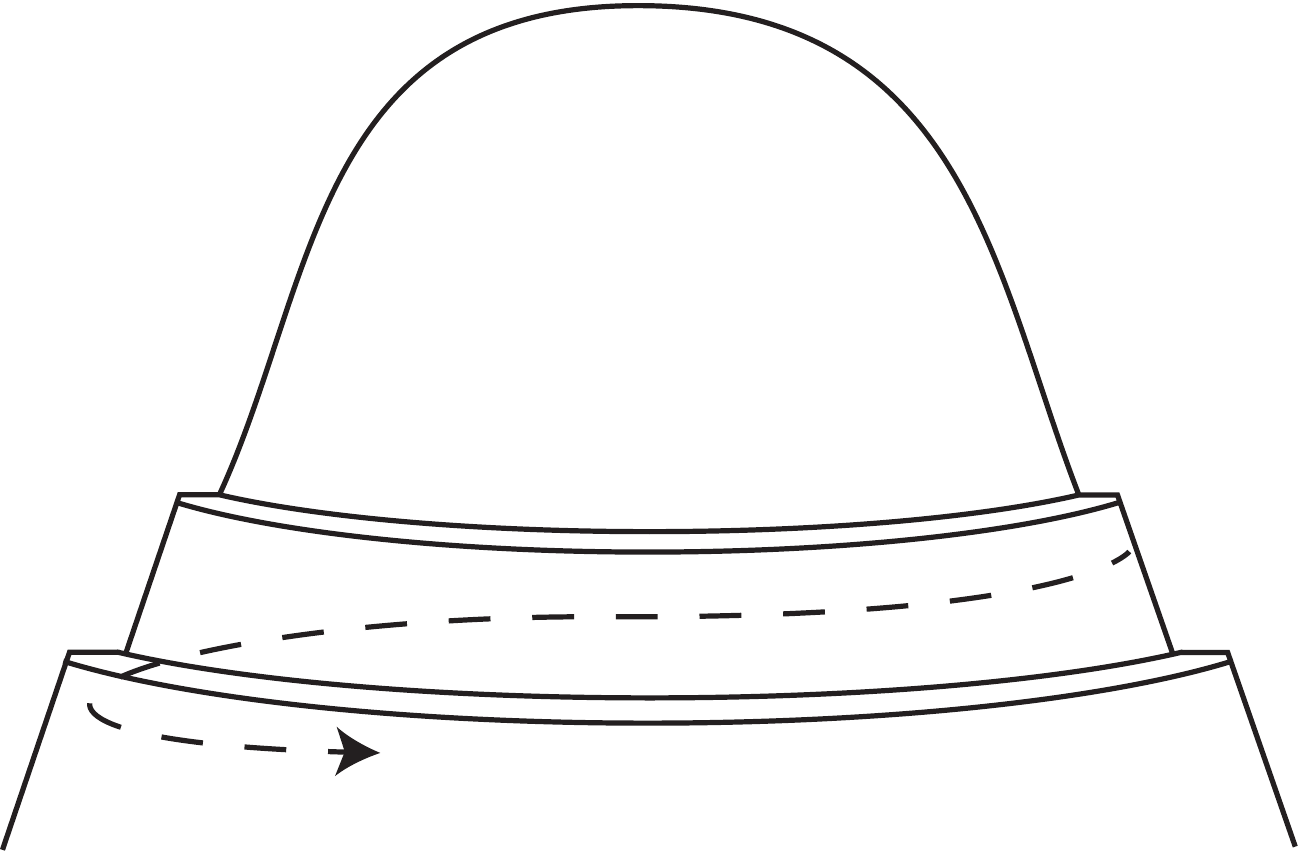}
\caption{There are smooth functions  vanishing on an open (compact) set for which the gradient flow lines spiral around the zero locus.  The flow lines have infinite length and the Lojasiewicz theorem fails.}   \label{f:1}
  \end{figure}

 \subsection{The Lojasiewicz Theorem}
 
 Next we will   explain how the second Lojasiewicz inequality is 
 typically used to show uniqueness.  
 Before we do that, observe first that in the second inequality we always work in 
 a small neighborhood of $p$ so that, in particular, $ |f(x)-f(p)|\leq 1$ and 
 hence smaller powers on the left hand side of the inequality imply the inequality for higher powers.  
 As it turns out, we will see that any positive power strictly less than one would do for uniqueness.

 Suppose   that   $f:\RR^n \to \RR$ is a differentiable function. 
 Let $x=x(t)$ be a curve in $\RR^n$ parametrized on $[0,\infty)$ whose velocity $x'=-\nabla f$. 
 We would like to show that if the second inequality of Lojasiewicz holds for $f$ with a power
 $1>\beta> 1/2$, then the Lojasiewicz theorem mentioned above holds. That is, if $x(t)$ has a limit 
 point $x_{\infty}$, then the length of the curve is finite and $\lim_{t\to \infty}x(t)=x_{\infty}$.  
Since $x_{\infty}$ is a 
 limit point of $x(t)$ and $f$ is non-increasing along the curve, $x_{\infty}$ 
must be a critical point for $f$.   

The  length of the curve $x(t)$ is $\int |\nabla f|$, so we must show that $\int_{1}^{\infty}|\nabla f| \,ds$ is finite. Assume that $f(x_{\infty})=0$ and note that if we 
set $f(t)=f(x(t))$, then $f'=-|\nabla f|^2$.  Moreover, by the second Lojasiewicz 
inequality, we get that $f' \leq - f^{2\beta}$ if $x(t)$ is sufficiently close to
$x_{\infty}$.  (Assume for simplicity below that $x(t)$ stays
in a small neighborhood $x_{\infty}$ for $t$ sufficiently large so that this inequality holds; the general case follows with trivial changes.)  Then this inequality can be rewritten as $(f^{1-2\beta})' \geq (2\beta-1)$ 
which integrates to
\begin{align}    \label{e:Ldecay}
  f(t) \leq C\, t^{\frac{-1}{2\beta-1}} \, .
\end{align}

We need to show that \eqr{e:Ldecay} implies that $\int_{1}^{\infty}|\nabla f| \,ds$ is finite.  
This shows that $x(t)$ converges to $x_{\infty}$ as $t\to \infty$.  
To see that $\int_{1}^{\infty}|\nabla f| \,ds$ is finite, observe  by the Cauchy-Schwarz inequality that
\begin{align}
  \int_{1}^{\infty}|\nabla f|\,ds =\int_{1}^{\infty}\sqrt{-f'}\,ds
    \leq \left(-\int_{1}^{\infty}f'\, s^{1+\epsilon}\,ds\right)^{\frac{1}{2}}
       \left(\int_{1}^{\infty} s^{-1-\epsilon}\,ds\right)^{\frac{1}{2}} \, .
\end{align}
It suffices therefore to show that
\begin{align}
-\int_{1}^{T} f'\, s^{1+\epsilon}\,ds
\end{align}
is uniformly bounded.  
Integrating by parts gives
\begin{align}
    \int_{1}^{T} f'\, s^{1+\epsilon}\,ds=|f\, s^{1+\epsilon}|^{T}_1-(1+\epsilon)
     \int_{1}^{T} f\, s^{\epsilon}\,ds\, .
\end{align}
If we choose $\epsilon> 0$ sufficiently small depending on $\beta$, 
then we see that this is bounded independent of $T$ and hence $\int_{1}^{\infty}|\nabla f|\,ds$ is finite.

 \subsection{Infinite dimensional Lojasiewicz inequalities and applications}
 
 Many problems in geometry and the calculus of variations are essentially questions about functionals on infinite dimensional spaces, such as the energy functional on the space of mappings or the area functional on the space of graphs over a hypersurface.
  Infinite dimensional versions of Lojasiewicz inequalities,  
 proven in a celebrated work of Leon Simon, \cite{Si1}, have played an important role in these areas over the last 30 years.   Clearly, the infinite dimensional inequalities have immediate applications to uniqueness of limits for gradient flows, but, perhaps surprisingly, they also have implications for singularities of nonlinear PDE's.

 \vskip2mm
Once singularities occur one naturally wonders what the singularities are like. A standard technique for analysing singularities is to magnify around them.  Unfortunately, singularities in many of the interesting problems in Geometric-PDE looked at under a microscope will resemble one blowup, but under higher magnification, it 
might (as far as anyone knows) resemble a completely different blowup. Whether this ever 
happens is perhaps the most fundamental question about singularities; see, e.g., \cite{Si2} and \cite{Hr}.    By general principles, the set of blowups is connected and, thus, the difficulty for uniqueness is when the blowups are not isolated in the space of blowups.  

One of the first major results 
on uniqueness was by Allard-Almgren in 1981, \cite{AA}, where uniqueness of tangent cones  with smooth 
 cross-section for minimal varieties is proven under an additional integrability assumption on the cross-section.  The integrability condition applies in a number of important cases, but it is difficult to check and is not satisfied in many other important cases.

Perhaps surprisingly,  blowups for a number of important Geometric PDE's can essentially be reformulated as infinite dimensional gradient flows of analytic functionals.  Thus, the uniqueness question would follow from an infinite dimensional version of Lojasiewicz's theorem for gradient flows of analytic functionals.   
 Infinite dimensional versions of Lojasiewicz inequalities were
 proven in a celebrated work of Leon Simon, \cite{Si1},
 for the area, energy,  and related functionals and
 used, in particular,  to prove a fundamental result about
 uniqueness of tangent cones with smooth cross-section of minimal surfaces.  This holds, for instance, at all singular points of an area-minimizing 
 hypersurface in $\RR^8$.
  
 Lojasiewicz inequalities follow easily near a critical point where the Hessian is uniformly non-degenerate
  (this is the infinite dimensional analog of a 
 non-degenerate critical point where the Hessian is full rank).  The difficulty is dealing with the directions in the kernel of the Hessian.  In the cases that Simon considers, the Hessian has finite dimensional kernel by elliptic theory.  The rough idea of his approach is to use the easy argument in the (infinitely many) directions where the Hessian is invertible and use the classical Lojasiewicz inequalities on the finite dimensional kernel.
 He makes this rigorous  by reducing the infinite
 dimensional version to the classical Lojasiewicz inequality using Lyapunov-Schmidt
  reduction.  Infinite dimensional  Lojasiewicz
 inequalities proven using  Lyapunov-Schmidt 
 reduction, as in the work of Simon,   have had a profound impact on various areas of analysis and geometry and are 
 usually referred to as Lojasiewicz-Simon inequalities.

  The cross-sections of the tangent cones at the singularities in these cases are   assumed to be smooth and compact and this is crucial.  This means that 
  nearby cross-sections can be written as graphs over 
  the cross-section and, thus, can be identified with functions on the cross-section of the cone.   The problem is then to prove a Lojasiewicz-Simon inequality for an analytic functional on a Banach space of functions, where $0$ is a critical point corresponding to the cross-section.

 

Uniqueness of tangents has important  applications to regularity of the singular set; see Section \ref{s:singular} and
cf., e.g., \cite{Si3}, \cite{Si4},  \cite{Si5}, \cite{BrCoL} and \cite{HrLi} and cf. Figure \ref{f:3}.

   \begin{figure}[htbp]		
\centering\includegraphics[totalheight=.4\textheight, width=.5\textwidth]{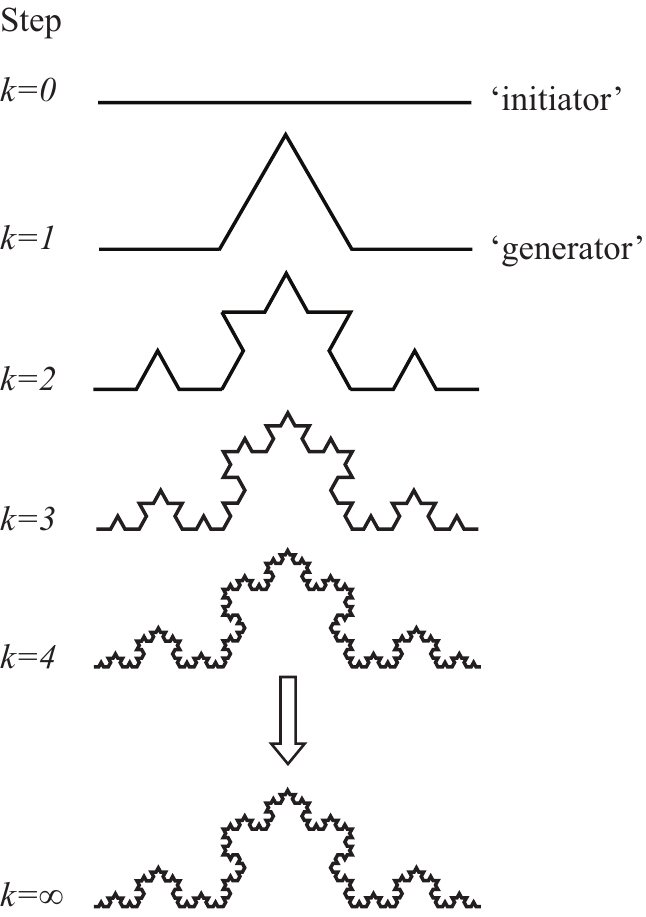}
\caption{The Koch curve is  close to a line on all scales, yet the line that it is close to changes from scale to scale.  It is not rectifiable but admits a H\"older parametrization.  It also illustrates that uniqueness of blowups is closely related to rectifiability.}   \label{f:3}
  \end{figure}

\section{Uniqueness of blowups for mean curvature flow}

In the next few sections, we will explain why at each generic singularity of a mean curvature flow the blowup is unique; that is independent of the sequence of rescalings; see Figure \ref{f:2}.  This very recent result settled a major open problem that was open even in the case of mean convex hypersurfaces where it was known that all singularities are generic.  Moreover, it is the first general uniqueness theorem for blowups to a Geometric-PDE at a non-compact singularity.  
 

   \begin{figure}[htbp]		
\centering\includegraphics[totalheight=.4\textheight, width=1\textwidth]{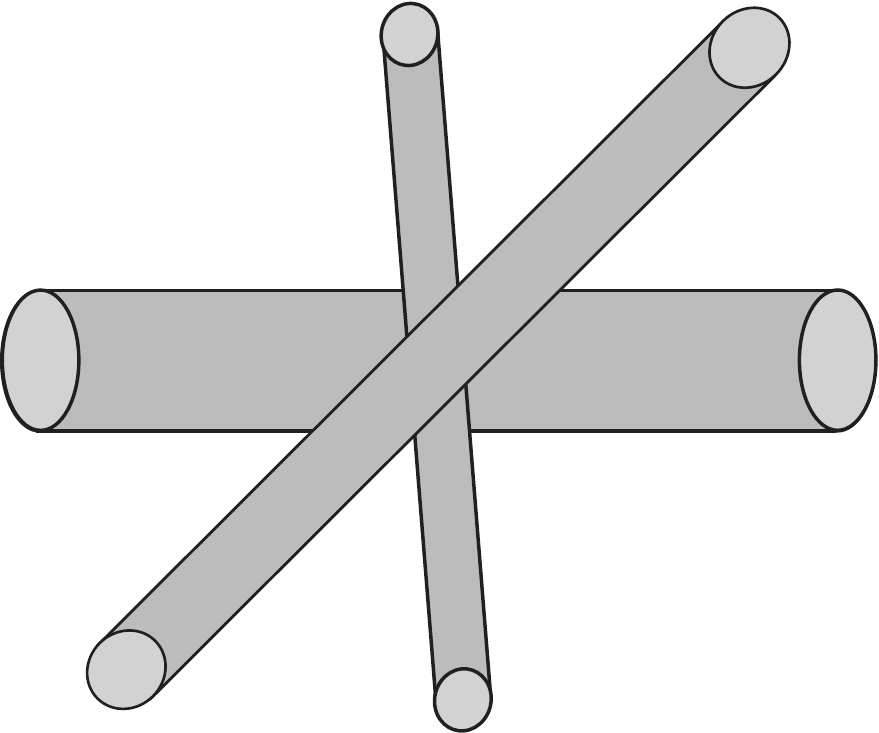}
\caption{The essence of uniqueness of tangent flows:  Can the flow be close to a cylinder at all times right before the singular time, yet the axis of the cylinder changes as the time gets closer to the singular time?}   \label{f:2}
  \end{figure}
 
 As already mentioned uniqueness of blowups is perhaps the most fundamental question that one can ask about singularities and  is known
 to imply regularity of the singular set.

 The proof of this uniqueness result
 relies on two completely new infinite dimensional Lojasiewicz type inequalities that, unlike all
 other infinite dimensional Lojasiewicz inequalities we know of, do not
 follow from   reduction to the classical finite-dimensional
 Lojasiewicz inequalities, but rather are proven directly and do not
 rely on Lojasiewicz's arguments or results.  
 
 It is well-known that to deal with non-compact singularities requires entirely new ideas and techniques as one cannot argue as in  Simon's work,
 and all the later work that  uses his ideas.  Partly because of this, it is expected that the techniques and ideas described here have applications to other flows.

\vskip2mm
The rest of this paper focuses on 
    mean curvature flow  (or MCF) of hypersurfaces.  This is  
a non-linear parabolic evolution equation where a hypersurface
evolves over time by locally moving in the direction of steepest
descent for the volume element.  It has been used and studied in material science for almost a century. 
Unlike some of the other earlier papers in material science, both von Neumann's 1952 paper and Mullins 1956 paper had explicit  equations.  In his paper von Neumann discussed soap foams whose interface tend to have constant mean curvature whereas Mullins is describing coarsening in metals, in which interfaces are not generally of constant mean curvature.  Partly as a consequence, Mullins may have been the first  to write down the MCF equation in general.  Mullins also found some of the basic self-similar solutions like the translating solution now known as the Grim Reaper.  
To be precise, suppose that
 $M_t\subset \RR^{n+1}$ is a one-parameter family of smooth hypersurfaces, then we say that $M_t$ flows by the  MCF if
\begin{equation}
x_t= -H\,\nn\, , 
\end{equation}
where $H$ and $\nn$ are   the mean curvature and unit normal, respectively, of 
$M_t$ at the point $x$.  
 
\subsection{Tangent flows}

By definition, a tangent flow is the limit of a sequence of rescalings
at a singularity, where the convergence is uniform on compact subsets.{\footnote{This is analogous to a tangent cone
at a singularity of a minimal variety, cf. \cite{FFl}.}}
For instance, a tangent flow to $M_t$ at the origin in space-time is the limit of 
a sequence of rescaled flows $\frac{1}{\delta_i} \, M_{\delta_i^2 \, t}$ where $\delta_i \to 0$.
A priori, different sequences $\delta_i$ could give different tangent flows
and the question of the uniqueness of the blowup - independent of the sequence - is a major
question in many geometric problems.
By a monotonicity formula of Huisken, \cite{H1}, and an argument of Ilmanen and
White, \cite{I},  \cite{W3}, tangent flows are   shrinkers, i.e., self-similar solutions of 
MCF that evolve by rescaling.
The only generic shrinkers are round cylinders by \cite{CM1}.

\vskip2mm 
We will say that a singular point is {\emph{cylindrical}} if at least one tangent flow is a multiplicity one cylinder $\SS^k \times \RR^{n-k}$.  
The main application of the new Lojasiewicz type inequality of \cite{CM2} is the 
following theorem that shows that  tangent flows at generic singularities are unique:

\begin{Thm} \label{t:main} 
\cite{CM2} 
Let $M_t$ be a  MCF in $\RR^{n+1}$.
  At each cylindrical singular point the tangent flow is unique.
  That is, any other tangent flow is also a cylinder with
  the same $\RR^k$ factor that points in the same direction.
\end{Thm}

This theorem solves a major open problem; see, e.g., page $534$ of \cite{W2}.  Even in the case of the evolution
of mean convex hypersurfaces where all singularities are cylindrical, uniqueness of the axis was unknown; see
\cite{HS1}, \cite{HS2}, \cite{W1}, \cite{W4}, \cite{SS}, \cite{An} and \cite{HaK}.\footnote{The results of \cite{CM2} not only give uniqueness of tangent flows but also a definite rate where the rescaled MCF converges to the relevant cylinder.  The distance to the cylinder is decaying to zero at a definite rate over balls whose radii are increasing at a definite rate to infinity.}

In recent joint work with Tom Ilmanen, \cite{CIM}, we showed
that if one tangent flow at a singular point of a MCF is
  a multiplicity one cylinder, then all are.  
  However, \cite{CIM} left open the possibility that the direction of the axis (the $\RR^k$ factor) 
  depended on the sequence of rescalings.  The proof of Theorem \ref{t:main} and, in particular, 
  the first Lojasiewicz type inequality of \cite{CM2}, 
  has its roots in some ideas and inequalities from \cite{CIM} and in fact implicitly use that cylinders are isolated among shrinkers by \cite{CIM}.
  
The  results of \cite{CM2} are the first general uniqueness theorems for tangent
flows to a geometric flow at a non-compact singularity.  (In fact, not only are the singularities that \cite{CM2} deal with non-compact but they are also non-integrable.)  Some special
cases of uniqueness of tangent flows for MCF were previously analyzed assuming either some sort of
convexity or that the hypersurface is a surface of rotation; see
\cite{H1}, \cite{H2}, \cite{HS1}, \cite{HS2}, \cite{W1}, \cite{SS}, \cite{AAG}, 
section $3.2$ in the book \cite{GGS}, and \cite{GK}, \cite{GKS}, \cite{GS}.  In contrast, uniqueness for
blowups at compact singularities is better understood; cf.\ \cite{AA},
  \cite{Si1}, \cite{H3}, \cite{Sc}, \cite{KSy}, and \cite{Se}.

One of the significant difficulties that \cite{CM2} overcomes, and sets it apart from all other work we know of, is that the singularities are noncompact.  This causes major analytical difficulties and to address them requires entirely new techniques and ideas. This is not so much because
of the  subtleties of analysis on noncompact domains, though this is an issue, but crucially
 because  the evolving hypersurface cannot be written as an entire graph over the singularity no matter how close we get to the singularity.
Rather,
the geometry of the situation dictates that  only part of the evolving hypersurface can be written as a graph 
over a compact piece of the singularity.{\footnote{In the end, what comes out of the analysis in \cite{CM2} is that the domain the evolving hypersurface is a graph over is expanding in time and at a definite rate, but this is not all all clear from the outset; see also footnote $3$.}}
 
\section{Lojasiewicz inequalities for non-compact hypersurfaces and MCF} 
The  infinite dimensional Lojasiewicz-type inequalities that \cite{CM2} showed are for the $F$-functional on the space of hypersurfaces.

The $F$-functional is given by integrating the Gaussian over a hypersurface $\Sigma\subset \RR^{n+1}$.  This is also often referred to as the Gaussian surface
area and is defined by
\begin{align}
  F (\Sigma) = (4\pi)^{-n/2} \, \int_{\Sigma} \,
  \e^{-\frac{|x|^2}{4}} \, d\mu \, .
\end{align}
The entropy $\lambda (\Sigma) $ is the supremum of the Gaussian surface areas
 over all centers and scales
 \begin{align}
  \lambda (\Sigma) = \sup_{t_0 > 0 , \, x_0 \in \RR^n  } \, \,   (4\pi \, t_0)^{-n/2} \, \int_{\Sigma} \,
  \e^{-\frac{|x-x_0|^2}{4t_0}} \, d\mu \, .
\end{align}
The entropy is a Lyapunov functional for both MCF and rescaled MCF (it is monotone non-increasing under the flows).

It follows from the first variation formula that the gradient of $F$ is 
\begin{align}	\label{e:gradF}
   \nabla_{\Sigma} F (\psi) = \int_{\Sigma} \, \left( H - \frac{\langle x , \nn \rangle}{2} \right)
   \, \psi \, \e^{ - \frac{|x|^2}{4} } \, .
\end{align}
Thus, the critical points of $F$ are shrinkers, i.e., hypersurfaces with $H = \frac{\langle x , \nn \rangle}{2} $.   
The most 
important shrinkers are the  generalized cylinders $\cC$; these are the generic ones by \cite{CM1}.
The space $\cC$ is the union of $\cC_k$ for $k \geq 1$, where 
 $\cC_{k}$ is the space of cylinders $\SS^k \times \RR^{n-k}$,
 where the $\SS^k$ is centered at $0$ and has radius $\sqrt{2k}$ and we allow all 
 possible rotations by $SO(n+1)$.

 \vskip2mm
A family of hypersurfaces $\Sigma_s$ evolves by the negative gradient flow for the $F$-functional 
if it satisfies  the equation 
\begin{align}
	(\partial_s x)^\perp= -H\,\nn +
x^\perp/2 \, .
\end{align}
 This flow is called the rescaled MCF since $\Sigma_s$ is 
  obtained   
from a MCF $M_t$ by setting $\Sigma_s=\frac{1}{\sqrt{-t}}M_t$,
$s=-\log(-t)$, $t<0$.  By \eqr{e:gradF},
critical points for the $F$-functional or, equivalently, stationary points for the rescaled 
MCF, are the shrinkers for the MCF that become extinct at the origin in space-time.    
A rescaled MCF has a unique asymptotic limit if and only if
the corresponding MCF has a unique tangent flow at that singularity.

\vskip2mm
The paper \cite{CM2} proved versions of the two Lojasiewicz inequalities for the $F$-functional on a general hypersurface $\Sigma$.  Roughly speaking, 
\cite{CM2} showed that
\begin{align}
   \dist (\Sigma , \cC)^2 &\leq C \, \left|\nabla_{\Sigma} F \right|\, ,   \label{e:07}  \\
    (F(\Sigma)-F(\cC))^{\frac{2}{3}} &\leq C \,\left|\nabla_{\Sigma} F \right| \, .   \label{e:08}
\end{align}
Equation \eqr{e:07} corresponds to Lojasiewicz's first inequality for $\nabla F$ whereas \eqr{e:08} corresponds to his second inequality for $F$.  The precise statements of these inequalities are much more complicated than this, but they are of the same flavor.   

\vskip2mm
As noted earlier a consequence of the classical Lojasiewicz gradient inequality for an analytical function on Euclidean space is that near a critical point there is no other critical values.  This consequence of a Lojasiewicz gradient inequality for the $F$-functional near a round cylinder (and in fact the corresponding consequence of \eqr{e:07}) was established in earlier joint work with Tom Ilmanen (see \cite{CIM} for the precise statement):

\begin{Thm}
\cite{CIM}   Any shrinker that is sufficiently  
close to a round cylinder on a large, but compact, set must itself be a round cylinder.
\end{Thm}

\vskip2mm
 In \cite{CM2} an infinite
 dimensional analog of the first Lojasiewicz inequality is proven directly and used   
 together with an infinite dimensional analog of the argument in Subsection \ref{ss:1to2} to show an analog of the 
  second Lojasiewicz inequality.   As mentioned, the reason why one cannot argue as in Simon's work, and all the later work that makes use his ideas, comes from that the singularities are noncompact.

\subsection{The two Lojasiewicz inequalities}  We will now state the two Lojasiewicz-type inequalities
for the $F$-functional on the space of hypersurfaces.

\vskip2mm
Suppose that $\Sigma\subset \RR^{n+1}$ is a hypersurface and fix
some sufficiently small $\epsilon_0 > 0$.   Given a large integer $\ell$ and a large constant $C_{\ell}$, we let $\graph (\Sigma)$ be the maximal radius so that
\begin{itemize}
 \item  $B_{\graph (\Sigma)} \cap \Sigma$ is the graph over a cylinder in
 $\cC_k$
 of a function $u$  with $\| u \|_{C^{2,\alpha}} \leq \epsilon_0$ and $|\nabla^{\ell} A| \leq C_{\ell}$.
\end{itemize}

In the next theorem,  we will use a Gaussian $L^2$ distance 
$d_{\cC}(R)$ to the space $\cC_k$ in the ball of radius $R$.    To define this, given $\Sigma_k  \in \cC_k$,  let $w_{\Sigma_k}: \RR^{n+1}\to \RR$ 
denote the distance to the axis of $\Sigma_k$ (i.e., to the space of translations that leave $\Sigma_k$ invariant).  Then we define
 \begin{align}
   d_{\cC}^2(R) = \inf_{\Sigma_k \in \cC_k} \,  \| w_{\Sigma_k} - \sqrt{2k} \|_{L^2(B_R)}^2 \equiv \inf_{  \Sigma_k \in \cC_k} \, 
      \int_{B_{R} \cap \Sigma_k} ( w_{\Sigma_k} - \sqrt{2k})^2\, \e^{-\frac{|x|^2}{4}}\, .
\end{align}
The Gaussian $L^p$ norm on the ball $B_R$ is
$\| u \|^p_{L^p(B_R)} = \int_{B_R} |u|^p \,  \e^{ - \frac{|x|^2}{4} }$.

Given a general hypersurface $\Sigma$, it is also convenient to define the function $\phi$ by
\begin{align}  
 \phi = \frac{\langle x , \nn \rangle}{2}-H \, ,
\end{align}
so that $\phi$ is minus the gradient of the functional $F$.

\vskip2mm
The main tools that \cite{CM2} developed are the
following two analogs for non-compact hypersurfaces of Lojasiewicz's  inequalities.    The first of these inequalities is really for the gradient whereas the second is for the function.

\begin{Thm}  \label{t:ourfirstloja}
(A Lojasiewicz inequality for non-compact hypersurfaces, \cite{CM2}). 
If $\Sigma \subset \RR^{n+1}$ is a hypersurface
with $\lambda (\Sigma) \leq \lambda_0 $  and $R \in [ 1 ,  \graph (\Sigma) - 1]$, then 
\begin{align}	\label{e:filoja}
	d^2_{\cC} (R) \leq     C \,  R^{\rho}\, \left\{ 
        \|  \phi \|_{L^1(B_{ R} )}^{ b_{\ell,n}  }   
       +  \e^{ - \frac{ b_{\ell , n} \, R^2}{4}   } \right\}  \, ,
\end{align}
where $C = C(n,\ell, C_{\ell}, \lambda_0)$,  $\rho = \rho (n)$ and  $b_{\ell , n} \in (0,1)$ satisfies $\lim_{\ell \to \infty} \, b_{\ell , n} = 1$.
\end{Thm}



\vskip2mm
The   theorem  bounds the $L^2$ distance to $\cC_k$ by a
power of  $\| \phi \|_{L^1}$, with an error term that comes from  
a cutoff argument since $\Sigma$ is non-compact and is not globally a 
graph of the cylinder.{\footnote{This is a Lojasiewicz inequality for 
the gradient of the $F$ functional ($\phi$ is the gradient of $F$).  This follows since, by \cite{CIM}, cylinders 
are isolated critical points for $F$ and, thus, $d_{\cC}$ 
locally measures the distance to the nearest critical point.}} 
This theorem is essentially sharp. 
Namely, the estimate \eqr{e:filoja} does not hold 
for any exponent $b_{\ell , n}$  
larger than one, but Theorem \ref{t:ourfirstloja} lets 
us take $b_{\ell , n}$ arbitrarily close to one.


\vskip2mm
In \cite{CM2} it is shown that the above inequality implies the following gradient type Lojasiewicz inequality.
  This inequality  bounds the difference of the $F$ functional near a critical point by two terms.  The first 
is essentially a power of $\nabla F$, while the second (exponentially decaying) term comes from
 that $\Sigma$ is not a graph over the entire cylinder.
 

 \begin{Thm}  \label{t:ourgradloja}
(A gradient Lojasiewicz inequality for non-compact hypersurfaces, \cite{CM2}).
If $\Sigma \subset \RR^{n+1}$ is a hypersurface
with $\lambda (\Sigma) \leq \lambda_0 $,  $\beta \in [0,1)$,  and $R \in [ 1 ,  \graph (\Sigma) - 1]$, then
\begin{align}	\label{e:fingloja} 
	 \left| F (\Sigma) - F (\cC_k) \right| 
      &\leq  C \,  R^{\rho}\, \left\{ 
        \|  \phi \|_{L^2(B_{ R} )}^{ c_{\ell,n} \, \frac{3+\beta}{2+2\beta} }   
       +  
         \e^{ - \frac{c_{\ell,n} \,(3+\beta) R^2}{8(1+\beta)}}    +  
	 \e^{ - \frac{(3+\beta)(R -1)^2}{16} }
       \,
        \right\}  \, ,
\end{align}
where $C = C(n,\ell, C_{\ell}, \lambda_0)$,  $\rho = \rho (n)$ and  $c_{\ell , n} \in (0,1)$ satisfies $\lim_{\ell \to \infty} \, c_{\ell , n} = 1$.
\end{Thm}

When the theorem is applied, the parameters $\beta$ and $\ell$ is chosen to make the exponent  greater
than one on  the $\nabla F$ term, essentially giving that $\left| F(\Sigma) - F(\cC_k) \right|$
is bounded by a power greater than one of $|\nabla F|$.  A separate argument is needed to handle the 
exponentially decaying error terms.

\vskip2mm
The paper \cite{CM2} showed that when $\Sigma_t$ are  flowing by the rescaled MCF, then both terms on the right-hand side of
\eqr{e:fingloja} are bounded by a power greater than one of $\| \phi \|_{L^2}$ (the corresponding statement   holds for 
Theorem \ref{t:ourfirstloja}).  Thus, one essentially get the inequalities
\begin{align}
d_{\cC}^2 &\leq C \, \left|\nabla_{\Sigma_t} F \right|\, ,\\
  (F(\Sigma_t)-F(\cC))^{\frac{2}{3}} &\leq C \,\left|\nabla_{\Sigma_t} F \right| \, .
\end{align}
These two inequalities can be thought of as analogs for the rescaled MCF of Lojasiewicz inequalities; cf. \eqr{e:07} and
\eqr{e:08}.

  \section{Cylindrical estimates for a general hypersurface}	\label{s:gensi}
  
 The proof of the two Lojasiewicz inequalities relies on some equations and estimates on general hypersurfaces  $\Sigma \subset \RR^{n+1}$.  Particularly important  are bounds for $\nabla \frac{A}{H}$
  when the mean curvature $H$ is positive on a large set.  This will be discussed in this section.  
  
 \subsection{A general Simons equation}

An important point for the proof of the Lojasiewicz type inequalities is that the second fundamental form $A$ of   $\Sigma $ satisfies an elliptic differential equation similar to Simons' equation for minimal surfaces.    The elliptic operator will be
the $L$ operator from \cite{CM1} given by
\begin{align}
	L \equiv \cL + |A|^2 + \frac{1}{2} \equiv \Delta - \frac{1}{2} \, \nabla_{x^T} + |A|^2 + \frac{1}{2} \, ,
\end{align}
where we have the following:

\begin{Pro}\cite{CM2}	\label{l:gensimons}
 If  $\phi = \frac{1}{2} \langle x , \nn \rangle - H$, then
 \begin{align}
    L \, A = A + \Hess_{\phi} + \phi \, A^2 \, , 
 \end{align}
where the tensor $A^2$ is given in orthonormal frame by $\left( A^2 \right)_{ij} = A_{ik} \, A_{kj}$.
\end{Pro}

Note that $\phi$ vanishes precisely when $\Sigma$ is a shrinker and, in this case, we recover the Simons' 
equation for $A$ for shrinkers from \cite{CM1}.

\subsection{An integral bound when the  mean curvature is positive}

One of the keys in the proof of the first Lojasiewicz type inequality is that the tensor $\tau = A/H$ is almost parallel when $H$ is positive and $\phi$ is small.  This generalizes an estimate from
\cite{CIM} in the case where $\Sigma$ is a shrinker (i.e., $\phi \equiv 0$) with $H> 0$.
  
  \vskip2mm
Given $f>0$, define a weighted divergence operator $\dv_f$ and drift
Laplacian $\cL_f$ by
\begin{align}
  \dv_f (V) &= \frac{1}{f} \, \e^{ |x|^2/4 } \, \dv_{\Sigma}
  \, \left( f \, \e^{ -|x|^2/4 } \, V \right) \, ,
  \\
  \cL_f \, u &\equiv \dv_f (\nabla u) =\cL \, u + \langle \nabla \log
  f , \nabla u \rangle \, .
\end{align}
Here $u$ may also be a tensor; in this case the divergence traces only
with $\nabla$. Note that $\cL=\cL_1$. We recall the quotient rule (see lemma $4.3$ in \cite{CIM}):

\begin{Lem} \label{l:quotr} Given a tensor $\tau$ and a function $g$
  with $g \ne 0$, then
  \begin{align}
    \cL_{g^2} \, \frac{\tau}{g} & = \frac{ g \, \cL \, \tau - \tau \,
      \cL \, g}{g^2} = \frac{ g \, L \, \tau - \tau \, L \, g}{g^2}
    \, .
  \end{align}
\end{Lem}

\begin{Pro}\cite{CM2} \label{p:eqnsg} 
 On the set where $H > 0$, we have
  \begin{align}
    \cL_{ H^2} \, \frac{ A}{H} &=    \frac{   \Hess_{\phi} + \phi \, A^2}{H}  + 
      \frac{A \, \left(   \Delta \phi + \phi \, |A|^2 \right) }{H^2} \, , \\
    \cL_{ H^2} \, \frac{ |A|^2}{H^2} &= 2\, \left|\nabla \frac{A}{H}
    \right|^2  + 2 \, \frac{   \langle    \Hess_{\phi} + \phi \, A^2  , A \rangle}{H^2}   + 2\, \frac{|A|^2 \, \left(   \Delta \phi + \phi \, |A|^2 \right)    }{H^3} \, .
  \end{align}
\end{Pro}

 The proposition follows easily from Proposition \ref{l:gensimons} and the Leibniz rule of Lemma \ref{l:quotr}; see \cite{CM2} for details.

\vskip2mm
The next proposition gives exponentially decaying integral bounds for
$\nabla (A/H)$ when $H$ is positive on a large ball.  It will be
important that these bounds decay rapidly.

\begin{Pro}\cite{CM2}   \label{p:effective} If $B_R \cap \Sigma$ is smooth with $H
  > 0$, then for $s \in (0 , R)$ we have
  \begin{align}
    \int_{B_{R-s} \cap \Sigma} \left| \nabla \frac{A}{H} \right|^2 \,
    H^2 \, \e^{-  \frac{|x|^2}{4} } &\leq \frac{4}{s^2} \, \sup_{B_R
      \cap \Sigma} |A|^2 \, \Vol (B_R \cap \Sigma) \, \e^{ -
      \frac{(R-s)^2}{4} } \notag  \\
      &\quad + 2\, \int_{B_R \cap \Sigma}  
      \left\{  \left|  \langle \Hess_{\phi} , A \rangle
         + \frac{|A|^2}{H} \, \Delta \phi \right|
      + \left| \langle A^2 , A \rangle  + \frac{|A|^4}{H} \right| \, |\phi|
      \right\}
      \e^{- \frac{|x|^2}{4} } \, .
  \end{align}
\end{Pro}

\begin{proof}
  Set $\tau= A/H$ and $u = | \tau |^2 = |A|^2/H^2$.  
     It will be convenient within this proof to use square brackets $\left[ \cdot \right]$ to denote Gaussian integrals 
  over $B_R \cap \Sigma$, i.e. $\left[ f \right] = \int_{B_R \cap \Sigma } f \, \e^{-|x|^2/4}$.

    Let $\psi$ be a function with support in
  $B_R$. Using the divergence theorem, the formula from  Proposition
  \ref{p:eqnsg}  for
  $\cL_{H^2}u$, and the absorbing inequality $4ab \leq a^2 + 4 b^2$, we get
  \begin{align}
    0 &= \left[  \dv_{H^2} \, \left( \psi^2 \, \nabla u \right) \, H^2  \left] 
    = \right[ \left(\psi^2 \, \cL_{H^2}u + 2 \psi
      \langle \nabla \psi , \nabla u \rangle \right)
    \, H^2 \right]  \notag   \\
    &= \left[ \left\{ 2\, \psi^2 \, \left| \nabla \tau \right|^2 + 2 \, \psi^2 \left( 
     \frac{   \langle    \Hess_{\phi} + \phi \, A^2  , A \rangle}{H^2}   +   \frac{|A|^2 \, \left(   \Delta \phi + \phi \, |A|^2 \right)    }{H^3}   \right) 
    + 4 \psi \langle \nabla \psi , \tau\cdot\nabla \tau
      \rangle \right\}
    \, H^2 \right]  \notag  \\
    &\geq \left[  \left(\psi^2 \, \left| \nabla \tau
      \right|^2- 4\, |\tau|^2 \, \left| \nabla \psi \right|^2 \right) \,
    H^2 \, \right]  + 2 \, \left[ \psi^2  \langle    \Hess_{\phi} + \phi \, A^2  , A \rangle \right] 
    +2 \, \left[ \psi^2 \frac{|A|^2 \, \left(   \Delta \phi + \phi \, |A|^2 \right)    }{H}  
    \right] \, , \notag
  \end{align}
  from which we obtain
  \begin{align}
    \left[ \psi^2 \, \left| \nabla \tau
      \right|^2    \, H^2 \right] 
   & \leq 4\, \left[ \left| \nabla \psi \right|^2 \, |A|^2 \, 
    \right]  - 2 \, \left[ \psi^2 \,  \langle  \Hess_{\phi} + \phi A^2, A \rangle    
     \right] 
    - 2 \, \left[\psi^2 \,  \Delta \phi \, \frac{|A|^2}{H} + \psi^2 \, \phi \frac{|A|^4}{H} \right] 
    \, . \notag 
  \end{align}
  The proposition follows by choosing $\psi \equiv 1$ on $B_{R-s}$ and
  going to zero linearly on $\partial B_{R}$.
\end{proof}

This proposition has the the following corollary:

\begin{Cor}\cite{CM2}    \label{c:effective} If $B_R \cap \Sigma$ is smooth with $H> \delta
  > 0$ and $|A| \leq C_1$, then there exists $C_2 = C_2 (n, \delta , C_1)$ so that for $s \in (0 , R)$ we have
  \begin{align}
     \int_{B_{R-s} \cap \Sigma} \left| \nabla \frac{A}{H} \right|^2  
   \, \e^{-  \frac{|x|^2}{4} } &\leq \frac{C_2}{s^2} \,  \Vol (B_R \cap \Sigma) \, \e^{ -
      \frac{(R-s)^2}{4} }  +  C_2 \, \int_{B_R \cap \Sigma}  \left\{   \left| \Hess_{\phi} \right| +  |  \phi|
         \right\}
      \e^{- \frac{|x|^2}{4} } \, .
  \end{align}
\end{Cor}

\begin{Rem}
Corollary \ref{c:effective} essentially bounds the distance {\emph{squared}} to the space of cylinders by $\|\phi \|_{L^1}$.
This is sharp: it is not possible to get the sharper
bound where the powers are the same.  This is a general fact when  there is a non-integrable kernel.  Namely,
 if we perturb  in the direction of the kernel, then  $\phi$ vanishes
 quadratically in the distance.
\end{Rem}

The next corollary combines the Gaussian $L^2$ bound on $\nabla \tau$ from 
Corollary \ref{c:effective}  with standard interpolation inequalities to get 
pointwise bounds on $\nabla \tau$ and $\nabla^2 \tau$.

\begin{Cor}\cite{CM2} 	\label{c:ptwise}
 If $B_R \cap \Sigma$ is smooth with $H> \delta
  > 0$,  $|A| + \left|\nabla^{\ell+1} A \right| \leq C_1$,  and $\lambda (\Sigma) \leq \lambda_0$, then 
  there exists $C_3 = C_3 (n , \lambda_0 , \delta , \ell , C_1)$ so that
   for   $|y| + \frac{1}{1+|y|} < R-1$,  we have
 \begin{align}
    \left|\nabla \, \frac{A}{H} \right| (y) + \left|\nabla^2  \frac{A}{H}  \right|(y)&\leq C_3 \, 
    R^{2n} \, \left\{ \e^{ - d_{\ell , n } \, \frac{(R-1)^2}{8} }+ \| \phi \|_{L^1(B_R)}^{ \frac{d_{\ell , n}}{2} } 
    \right\}  \,    \e^{ \frac{|y|^2}{8} } \, , 
 \end{align}
where the exponent $d_{\ell,n} \in (0,1)$ has $\lim_{\ell \to \infty} d_{\ell,n} 
 =1$.

\end{Cor}

 See \cite{CM2} for the proof of Corollary \ref{c:ptwise}.
 
\section{Distance to  cylinders and the first Lojasiewicz inequality}	\label{s:dcyl}

 Finally, we will briefly outline how one get from Corollary \ref{c:ptwise} to the proof of the first Lojasiewicz type inequality for the $F$-functional.    This inequality will follow from the  bounds on the  
 tensor $\tau = \frac{A}{H}$ in the previous section  together with the following proposition:
 
\begin{Pro}\cite{CM2}	\label{p:cylclose}
Given $n$, $\delta > 0$ and $C_1$, there exist 
  $\epsilon_0 > 0$, $\epsilon_1 > 0$ and $C$ so that
if
$\Sigma \subset \RR^{n+1}$ is a hypersurface (possibly with boundary)
that satisfies:
\begin{enumerate}
 \item $H \geq \delta > 0$ and $|A| + |\nabla A| \leq C_1$ on $B_R \cap \Sigma$.
 \item $B_{5\sqrt{2n}} \cap \Sigma$ is $\epsilon_0$ $C^{2}$-close to a cylinder
 in $\cC_k$ for some $k \geq 1$,
\end{enumerate}
then, for any $r \in (5 \sqrt{2n} , R)$ with
\begin{align}
	r^2 \, \sup_{B_{5\sqrt{2n}}} \left( |\phi| + |\nabla \phi | \right) + r^5 \, \sup_{B_r}  \left( |\nabla \tau | + |\nabla^2 \tau | \right)
	\leq \epsilon_1  \, , 
\end{align}
we have that 
$B_{\sqrt{r^2-3k}} \cap \Sigma$ is the graph over (a subset of) a cylinder in $\cC_k$
of   $u$ with
\begin{equation}
    |u| + |\nabla u| \leq C \, 
    \left\{ r^2 \, \sup_{B_{5\sqrt{2n}}} \left( |\phi| + |\nabla \phi | \right) + r^5 \, \sup_{B_r} \left( |\nabla \tau | + |\nabla^2 \tau | \right) \right\}  \, .
\end{equation}
 
\end{Pro}

\vskip2mm
This proposition shows that $\Sigma$ must be close to a cylinder as long as $H$ is positive, $\phi$ is small,  $\tau$ is almost parallel and $\Sigma$ is close to a cylinder on a fixed small ball.  Together with Tom Ilmanen, we proved a similar result in proposition $2.2$ in \cite{CIM} in the special case where $\Sigma$ is a shrinker (i.e., when $\phi \equiv 0$) and this proposition is inspired by that one.

\vskip2mm
To prove Proposition \ref{p:cylclose} we make use of 
  the following result from
\cite{CIM} (see corollary $4.22$ in \cite{CIM}):

\begin{Lem}\cite{CIM}  \label{c:lowev} If $\Sigma \subset \RR^{n+1}$ is
  a hypersurface (possibly with boundary) with  
 \begin{itemize}
  \item $0 <  \delta \leq H$ on $\Sigma$,
  \item the tensor $\tau \equiv A/H$ satisfies $\left| \nabla \tau
    \right| + \left| \nabla^2 \tau \right| \leq \eps\leq 1$,
  \item At the point $p \in \Sigma$, $\tau_p$ has at least two distinct eigenvalues $\kappa_1 \ne \kappa_2$,
 \end{itemize}
then  
\begin{align*}
  \left| \kappa_1 \kappa_2 \right| \leq \frac{2 \, \eps}{\delta^2}
  \, \left(  \frac{1}{\left| \kappa_1 - \kappa_2 \right|}  +
  \frac{1}{\left| \kappa_1 - \kappa_2 \right|^2} \right) \, .
\end{align*}

\end{Lem}

From this lemma, we see that if the assumption of the lemma holds  for a  hypersurface, then the principal curvatures divide into two groups.   One group consists of principal curvatures that are close to zero and the other group consists of principal curvatures that cluster around a non-zero real number.  Thus, we get flatness for any two-plane containing a  principal direction  in the first group, while  any two-plane spanned by principal directions in the second group  is umbilic.  This is the starting point for the proof of Proposition \ref{p:cylclose}.

  \section{The singular set of MCF with generic singularities}	\label{s:singular}
 
 A major theme in PDE's over the last fifty years has been understanding singularities and the set where singularities occur. In the presence of a scale-invariant monotone quantity, 
blowup arguments can often be used to bound the dimension of the 
singular set; see, e.g., \cite{Al}, \cite{F}.
Unfortunately, these dimension bounds say little about the 
structure of the set.  However, using the results of the previous sections, \cite{CM3}
gave a rather complete description of the singular set for MCF with generic singularities.

\vskip2mm
The main result of \cite{CM3} is the following: 

\begin{Thm}[\cite{CM3}]	\label{t:main51}
Let $M_t \subset \RR^{n+1}$ be a MCF of  closed embedded hypersurfaces with only cylindrical singularities, then the space-time singular set satisfies:
\begin{itemize}
\item It is contained in finitely many (compact) embedded Lipschitz\footnote{In fact, Lipschitz is with respect to the parabolic  distance  on space-time which is a much stronger assertion than Lipschitz with respect to the Euclidean distance.  Note that a function  is Lipschitz when the target has the parabolic metric on $\RR$ is equivalent to that it is $2$-H\"older for the standard metric on $\RR$.} submanifolds each of dimension at most $(n-1)$ together with a set of dimension at most $(n-2)$.

\item It consists of countably many graphs of $2$-H\"older functions on
space.
\item The time image of each subset with finite parabolic $2$-dimensional Hausdorff measure has measure zero; each such connected subset is contained in a time-slice.
\end{itemize}
\end{Thm}

In fact, \cite{CM3}  proves considerably more than what is stated in Theorem \ref{t:main51}; see  theorem  $4.18$ in \cite{CM3}.  For instance, instead of just proving the first claim of the theorem,   the entire stratification of the space-time singular set is Lipschitz of the appropriate dimension.  Moreover, this holds without ever discarding {\bf{any}} subset of measure zero of any dimension  as is always implicit in any definition of rectifiable.  To illustrate the much stronger version, consider the case of evolution of surfaces in $\RR^3$.  In that case, this gives that the space-time singular set is contained in finitely many (compact) embedded Lipschitz curves with cylinder singularities together with a countable set of spherical singularities.   In higher dimensions,   the direct generalization of this is proven.

 \vskip2mm
Theorem \ref{t:main51} has the following corollaries:

\begin{Cor}[\cite{CM3}]	\label{c:main}
Let $M_t \subset \RR^{n+1}$ be a MCF of closed embedded mean convex hypersurfaces or a MCF with only generic singularities, then the conclusion of Theorem \ref{t:main} holds.
\end{Cor}

More can be said in dimensions three and four:

\begin{Cor}[\cite{CM3}]  \label{c:main2}
If $M_t$ is as in Theorem \ref{t:main51} and $n=2$ or $3$, then the evolving hypersurface is completely smooth (i.e., without any singularities) at almost all times.   
In particular, any connected subset of the space-time singular set is completely contained in a time-slice.  
\end{Cor}

\begin{Cor}[\cite{CM3}]  \label{c:main3}
For a generic MCF in $\RR^3$ or $\RR^4$ or a flow starting at a closed embedded mean convex hypersurface in $\RR^3$ or $\RR^4$ the conclusion of Corollary \ref{c:main2} holds.  
\end{Cor}

The conclusions of Corollary \ref{c:main3} hold in all dimensions if  the initial   hypersurface is $2$- or $3$-convex.   A hypersurface is said to be $k$-convex if the sum of any $k$ principal curvatures is nonnegative.

\vskip2mm
A key technical point in \cite{CM3} is to prove a strong parabolic Reifenberg property for MCF with generic singularities.  In fact,   the space-time singular set is proven to be (parabolically) Reifenberg vanishing.  In Analysis a subset of Euclidean space is said to be Reifenberg (or Reifenberg flat) if on all sufficiently small scales it is, after rescaling to unit size close, to a $k$-dimensional plane.  The dimension of the plane is always the same but the plane itself may change from scale to scale.  Many snowflakes, like the Koch snowflake, are Reifenberg with Hausdorff dimension strictly larger than one.  A set is said to be Reifenberg vanishing if the closeness to a $k$-plane goes to zero as the scale goes to zero.  It is said to have the strong Reifenberg property if the $k$-dimensional plane depends only on the point but not on the scale.  Finally, one sometimes distinguishes between half Reifenberg and full Reifenberg, where half Reifenberg refers to that the set is close to a $k$-dimensional plane, whereas full Reifenberg refers to that in addition one also has the symmetric property: The plane on the given scale is close to the set.  

Using the results from \cite{CM2} described earlier in this paper,    \cite{CM3} shows that the singular set in space-time is strong (half) Reifenberg vanishing with respect to the parabolic Hausdorff distance.   This is done in two steps, showing   first that nearby singularities sit inside a parabolic cone (i.e., between two oppositely oriented space-time paraboloids that are tangent to the time-slice through the singularity). In fact,  
this   parabolic cone property holds with vanishing constant.  Next, in the complementary region of the parabolic cone in space-time (that is essentially space-like),  the parabolic Reifenberg essentially follows from the space Reifenberg that the uniqueness of \cite{CM2} of tangent flows implies.   

An immediate consequence, of independent interest, of our parabolic cone property with vanishing constant is that nearby a generic singularity in space-time (nearby is with respect to the parabolic distance) all other singularities happen at almost the same time.  

\vskip4mm
These results  should be contrasted with a result of Altschuler-Angenent-Giga, \cite{AAG} (cf. \cite{SS}), that shows that in $\RR^3$ the evolution of any rotationally symmetric surface obtained by rotating the graph of a function $r = u(x)$, $a < x < b$ around the $x$-axis is smooth except at finitely many singular times where either a cylindrical or spherical singularity forms.  For more general rotationally symmetric surfaces (even mean convex), the singularities can consist of nontrivial curves.  For instance, consider a torus of revolution bounding a region $\Omega$. If the torus is thin enough, it will be mean convex.  Since the symmetry is preserved and because the surface always remains in $\Omega$, it can only collapse to a circle. Thus at the time of collapse, the singular set is a simple closed curve.   White showed that a mean convex surface evolving by MCF in $\RR^3$ must be smooth at almost all times, and at no time can the singular set be more than $1$-dimensional (see section $5$ in \cite{W2}).   
In all dimensions, White 
showed that the space-time singular set of a mean convex MCF has parabolic Hausdorff dimension at most $(n-1)$; see  \cite{W1} and cf.   theorem $1.15$ in \cite{HaK}.
In fact, White's general dimension reducing argument  gives that the singular set of any MCF with only cylindrical singularities has dimension at most $(n-1)$.

\vskip2mm
These results motivate  the following conjecture:

\begin{Con}[\cite{CM3}]
Let $M_t$ be a MCF of closed embedded hypersurfaces in $\RR^{n+1}$ with only cylindrical singularities.  Then the 
space-time singular set has only finitely many components.  
\end{Con}

If this conjecture was true, then it would follow from this paper that in $\RR^3$ and $\RR^4$ mean curvature flow with only generic singularities is smooth except at finitely many times; cf. with the three-dimensional conjecture at the end of section $5$ in \cite{W2}.

\end{document}